\documentclass[11pt]{amsart}
\usepackage[margin=1.35in]{geometry}
\usepackage{amscd,amsmath,amsxtra,amsthm,amssymb,stmaryrd,xr,mathrsfs,mathtools,enumerate,commath, comment}
\usepackage{stmaryrd}
\usepackage{xcolor}
\usepackage{commath}
\usepackage{comment}
\usepackage{tikz-cd}
\usepackage{longtable} 
\usepackage{pdflscape} 
\usepackage{booktabs}
\usepackage{hyperref}
\definecolor{vegasgold}{rgb}{0.77, 0.7, 0.35}
\definecolor{darkgoldenrod}{rgb}{0.72, 0.53, 0.04}
\definecolor{gold(metallic)}{rgb}{0.83, 0.69, 0.22}
\hypersetup{
 colorlinks=true,
 linkcolor=darkgoldenrod,
 filecolor=brown,      
 urlcolor=gold(metallic),
 citecolor=darkgoldenrod,
 pdftitle={On large Iwasawa $\lambda$-invariants of imaginary quadratic function fields},
 }

\usepackage[all,cmtip]{xy}

\DeclareFontFamily{U}{wncy}{}
\DeclareFontShape{U}{wncy}{m}{n}{<->wncyr10}{}
\DeclareSymbolFont{mcy}{U}{wncy}{m}{n}
\DeclareMathSymbol{\Sh}{\mathord}{mcy}{"58}
\usepackage[T2A,T1]{fontenc}
\usepackage[OT2,T1]{fontenc}

\newtheorem{theorem}{Theorem}[section]
\newtheorem{lemma}[theorem]{Lemma}

\newtheorem{corollary}[theorem]{Corollary}

\newtheorem{proposition}[theorem]{Proposition}

\newcommand{\cF}{\mathcal{F}}

\newcommand{\Z}{\mathbb{Z}}

\newcommand{\F}{\mathbb{F}}

\newcommand{\cO}{\mathcal{O}}

\newcommand{\X}{\mathfrak{X}}

\newcommand{\op}[1]{\operatorname{#1}}

\numberwithin{equation}{section}

\begin{document}

\title[Large Iwasawa $\lambda$-invariants of imaginary quadratic function fields]{On large Iwasawa $\lambda$-invariants of imaginary quadratic function fields}

\author[A.~Ray]{Anwesh Ray}
\address[A.~Ray]{Centre de recherches mathématiques,
Université de Montréal,
Pavillon André-Aisenstadt,
2920 Chemin de la tour,
Montréal (Québec) H3T 1J4, Canada}
\email{anwesh.ray@umontreal.ca}

\maketitle

\begin{abstract}

Let $\ell$ be a prime number and $q$ be a power of $\ell$. Given an odd prime number $p$ and an imaginary quadratic extension $F$ of the rational function field $\mathbb{F}_q(T)$, let $\lambda_p(F)$ denote the Iwasawa $\lambda$-invariant of the constant $\mathbb{Z}_p$-extension of $F$. We show that for any number $r>0$ and all large enough values of $q\not\equiv 1\mod{p}$, there is a positive proportion of imaginary quadratic fields $F/\F_q(T)$ with the property that $\lambda_p(F)\geq r$. The main result is proved as a consequence of recent unconditional theorems of Ellenberg-Venkatesh-Westerland on the distribution of class groups of imaginary quadratic function fields.
\end{abstract}

\section{Introduction}
Let $\ell$ be a prime number and $q$ be a power of $\ell$. Denote by $\F_q(T)$ the field of rational functions over $\F_q$. An \emph{imaginary quadratic field} extension of $\F_q(T)$ is a degree $2$ extension of $\F_q(T)$ in which the prime $\infty$ is ramified. The question studied in this note is motivated by recent results of Ellenberg, Venkatesh and Westerland \cite{ellenberg2016homological} on the distribution of class groups of imaginary quadratic field extensions of $\F_q(T)$. The story begins with predictions made for the distribution of class groups for imaginary quadratic number fields made by Cohen and Lenstra \cite{cohen1984heuristics}, based in random matrix theory. The work of Cohen-Lenstra led to significant developments in the field of \emph{arithmetic statistics}. In the number field context, the predictions are far from proven unconditionally. We consider function field analogues of such heuristics in positive characteristic. The aforementioned results of Ellenberg-Venkatesh-Westerland show that the Cohen-Lenstra heuristics in the function field context are true in the \emph{large $q$-limit}. We refer the reader to \emph{loc. cit.} or Theorem \ref{EVW thm} for a precise statement of the result. 

\par The purpose of this note is to shed light on Iwasawa theoretic invariants introduced and systematically studied by Leitzel \cite{leitzel1970class} and Rosen \cite[Chapter 11]{rosen2002number}. Let $F$ be an imaginary quadratic field extension of $\F_q(T)$ and $p$ be an odd prime. The prime $p$ will be fixed, however, $q$ will be required to be large for our results to hold. For convenience of notation, set $\kappa:=\F_q$ and choose an algebraic closure $\bar{\kappa}$. Denote by $\Z_p$ the ring of $p$-adic integers and $\kappa_\infty^{(p)}$ be the unique $\Z_p$-extension of $\kappa$ which is contained in $\bar{\kappa}$. The constant $\Z_p$-extension $F_\infty$ of $F$ is the composite $F_\infty:=F\cdot \kappa_\infty^{(p)}$. It is constant in the sense that it arises from a subextension of $\bar{\kappa}$. Given $n\in \Z_{\geq 0}$, let $F_n$ be the extension of $F$ which is contained in $F_\infty$ with Galois group $\op{Gal}(F_n/F)$ isomorphic to $\Z/p^n\Z$. The main focus is the study of growth patterns of $p$-class groups in the tower of extensions 
\[F=F_0\subset F_1\subset F_2\subset \dots\subset F_n \subset F_{n+1}\subset \dots .\] Let $e_n$ be the largest power of $p$ that divides the class number of $F_n$. Leitzel and Rosen show that there are well-defined invariants $\lambda_p(F)\in \Z_{\geq 0}$ and $\nu_p(F)\in \Z$ such that for all large enough values of $n$,
\[e_n=\lambda_p(F)n +\nu_p(F).\] In particular, the analogue of the \emph{Iwasawa $\mu$-invariant} vanishes. We fix an odd prime $p$ and study the distribution of $\lambda_p(F)$, where $F$ is allowed to vary over all imaginary quadratic function field extensions of $\F_q(T)$. The main result shows that given any fixed odd prime number $p$ and integer $r>0$, there is a constant $c$ depending only on $p$ and $r$ such that for all $q\geq c$, there is a positive proportion of imaginary quadratic field extensions $F$ of $\F_q(T)$ for which $\lambda_p(F)\geq r$. This result is entirely unconditional, we refer to Theorem \ref{main theorem} for a more detailed statement of the result.
\par We mention here that in the number field context some related questions for $\lambda$-invariants of imaginary quadratic fields were studied by Horie \cite{horie1987note}, Jochnowitz \cite{jochnowitz1994p}, Sands \cite{sands1993non}, Ellenberg-Jain-Venkatesh \cite{ellenberg2011modeling}, Delbourgo-Knospe \cite{delbourgo2022iwasawa}, and the author of this note \cite{ray2022}. We draw our attention in particular to the above mentioned work of Sands, who shows unconditionally that given an odd prime number $p$, there are infinitely many imaginary quadratic number fields $F$ for which the $p$-primary $\lambda$-invariant $\lambda_p(F)$ is $\geq 2$. In the function field context a much stronger result is obtained in this note.

\emph{Organization:} Including the introduction, this note has four sections. In section \ref{s 2}, we introduce preliminary notions from Iwasawa theory and function field arithmetic. In section \ref{s 3}, we prove a number of consequences for the Iwasawa theory of class groups in constant $\Z_p$-towers of function fields. Finally, in section \ref{s 4}, we prove the main result, i.e., Theorem \ref{main theorem}. 

\subsection*{Acknowledgment} The author's research is supported by the CRM Simons postdoctoral fellowship. The author thanks the anonymous referee for a thorough and timely review.

\section{Preliminaries}\label{s 2}

\par In this section, we set up some basic notation and introduce preliminary notions in function field arithmetic and Iwasawa theory of constant $\Z_p$-extensions. For an introduction to function field arithmetic, we refer to \cite[Chapter 5]{rosen2002number}. For an introduction to the cyclotomic Iwasawa theory of number fields, we refer to \cite[Chapter 13]{washington1997introduction}, and for function field analogues in constant $\Z_p$-extensions, we refer to \cite[Chapter 11]{rosen2002number}.

\subsection{Function fields and their class groups} 
\par Throughout, we fix a prime number $p$ and a prime number $\ell$ (that are not necessarily distinct). We assume without further mention that $p$ is odd, however, $\ell=2$ is allowed. Let $\Z_p$ denote the ring of $p$-adic integers, i.e., the inverse limit $\varprojlim_n \Z/p^n\Z$. The Iwasawa theory of function fields is concerned with growth patterns in the $p$-primary parts of class groups in certain \emph{$\Z_p$-towers} over a function field, which we now introduce. 

\par Let $F$ be a global function field of characteristic $\ell$ and field of constants $\kappa=\F_q$. This means that $\kappa$ is algebraically closed in $F$ and the field $F$ is a finite extension of $\kappa(T)$. We note here in passing that the terminology stems from algebraic geometry, since $F$ is the function field of a smooth, geometrically integral, projective algebraic curve defined over $\kappa$. There is a beautiful analogy between number fields and function fields, cf. \cite{rosen2002number} for further details.

\par According to \cite[Chapter 5]{rosen2002number}, a \emph{prime} of $F$ is defined to be the maximal ideal of a discrete valuation ring $R$ which is contained in $F$, and has fraction field $F$. For instance, for $a\in \kappa$, the corresponding prime of $\kappa(T)$ is the maximal ideal of $\kappa[T]_{(T-a)}$. This corresponds to the point $a\in \mathbf{P}^1(\kappa)$. The prime $\infty$ corresponds to the maximal ideal of $\kappa[1/T]_{(1/T)}$. The primes of $\kappa(T)$ which correspond in this way to points of $\mathbf{P}^1(\kappa)$ are called \emph{rational primes}, since they correspond to $\kappa$-valued rational points. Let $P$ be a prime of $F$, and $R$ be the associated discrete valuation ring. The \emph{degree of $P$} is the dimension of $R/P$ over $\kappa$. The prime $P$ is said to be \emph{rational} if has degree $1$. In this case, the prime corresponds to a $\kappa$-valued rational point on the associated projective curve $\mathfrak{X}$. More generally, a prime corresponds to a Galois orbit in $\mathfrak{X}(\bar{\kappa})$, and its degree is the cardinality of its orbit with respect to the action of $\op{Gal}(\bar{\kappa}/\kappa)$.

\par A \emph{divisor} of $F$ is a formal integral linear combination of primes of $F$. Two divisors are equivalent if they differ by a \emph{principal divisor}. The \emph{class group} $\op{Cl}(F)$ is the group of all equivalence classes of divisors of degree $0$. We refer to \emph{loc. cit.} for more precise definitions. The group $\op{Cl}(F)$ is finite (cf. \cite{rosen2002number}) and the class number $h(F)$ is the number of elements in $\op{Cl}(F)$. Let $\op{Cl}_p(F)$ be the $p$-Sylow subgroup of $\op{Cl}(F)$ and set $h_p(F):=\# \op{Cl}_p(F)$. 
\par An imaginary quadratic function field $F$ is a function field with field of constants $\kappa$ such that
\begin{itemize}
    \item there is a prescribed inclusion of $\kappa(T)$ into $F$, 
    \item the degree of $F$ over $\kappa(T)$ is equal to $2$,
    \item $\infty$ is ramified in $F$.
\end{itemize}
Let $F$ be an imaginary quadratic field and denote by $\infty_F$ the prime of $F$ that lies above $\infty$. The ring of integers of $F$ is the integral closure of the polynomial ring $\kappa[T]$ in $F$ and is denoted by $\cO_F$. It is easy to see that $\cO_F$ consists of the functions in $F$ with no poles away from $\{\infty\}$, and thus $\cO_F$ is the ring of $S$-integers for $S=\{\infty\}$. Let $\op{Cl}(\cO_F)$ denote the class group associated to $\cO_F$ (cf. \cite[p.62]{halter1990note} for further details). Note that there is a surjective map $\op{Cl}(F)\rightarrow \op{Cl}(\cO_F)$ with kernel generated by the divisor class of $\infty_F$ (see \emph{loc. cit.}). Thus for any odd prime $p$, denote by $\op{Cl}_p(\cO_F)$ the $p$-Sylow subgroup of $\op{Cl}(\cO_F)$; we find that $\op{Cl}_p(\cO_F)$ is naturally isomorphic to $\op{Cl}_p(F)$. Let $L/F$ be a finite extension in which $\infty_F$ is totally inert. Identify the prime in $L$ above $\infty_F$ with $\infty_F$ itself. We find that for any odd prime $p$, $\op{Cl}_p(\cO_L)$ is naturally isomorphic to $\op{Cl}_p(L)$. The Hilbert class field $H(L)$ is defined to be the maximal abelian unramified extension of $L$ in which $\infty_F$ is totally split. There is a natural isomorphism $\op{Gal}(H(L)/L)\simeq \op{Cl}(\cO_L)$, cf. \cite{rosen1987hilbert, halter1990note}. Let $p$ be an odd prime number and $H_p(L)$ denote the maximal abelian pro-$p$ unramified extension of $L$ in which $\infty_F$ is completely split. Then by the above discussion, there is a natural isomorphism 
\[\op{Gal}(H_p(L)/L)\simeq \op{Cl}_p(L). \]

\subsection{Class group towers over constant $\Z_p$-extensions}
\par Let $F$ be an imaginary quadratic field and $\infty_F$ be the prime above $\infty$. The field of constants is $\kappa$, and $\bar{\kappa}$ is the algebraic closure of $\kappa$ in $\bar{F}$. For $n\in \Z_{\geq 1}$, let $\kappa_n\subset \bar{\kappa}$ be the constant extension of $\kappa$ such that $\op{Gal}(\kappa_n/\kappa)$ is isomorphic to $\Z/n\Z$. Recall that $p$ is an odd prime which is not necessarily distinct from $\ell$. Let $\Z_p$ be the ring of $p$-adic integers and set $\kappa_n^{(p)}:=\kappa_{p^n}$, $F_n^{(p)}:=F\cdot \kappa_n^{(p)}$. When there is no cause for confusion, we shall simply set $F_n$ to denote $F_n^{(p)}$. The tower of function field extensions \[F=F_0\subset F_1\subset F_2\subset \dots \subset F_n\subset \] is called the \emph{constant $\Z_p$-tower} over $F$. The union $F_\infty=\bigcup_n F_n$ is the \emph{constant $\Z_p$-extension} of $F$. We note that the Galois group $\Gamma:=\op{Gal}(F_\infty/F)$ is isomorphic to $\Z_p$. Furthermore, $\infty$ is totally inert in $F_\infty$. Let $L_n$ denote the $p$-Hilbert class field $H_p(F_n)$ and set $L_\infty:=\cup_{n\geq 1} L_n$. Denote by $X_n$ the Galois group $\op{Gal}(L_n/F_n)$ and set $\X:=\op{Gal}(L_\infty/F_\infty)$. The group $X_n$ is naturally identified with the $p$-primary class group $\op{Cl}_p(F_n)$. Given $m\geq n$, there is a natural map $X_m\rightarrow X_n$, which is given by the following composite
\[X_m\rightarrow \op{Gal}(L_n\cdot K_m/K_m)\rightarrow \op{Gal}(L_n/ L_n\cap K_m)\rightarrow X_m.\]
Since $\infty$ is completely split in the extension $L_n/K_n$ and it is completely inert in $K_m/K_n$, it follows that $L_n\cap K_m=K_n$. Therefore, the map $X_m\rightarrow X_n$ is surjective.

\par Denote by $\Gamma$ the Galois group $\op{Gal}(F_\infty/F)$ and fix a topological generator $\gamma$ in $\Gamma$. The Iwasawa algebra is defined as follows
\[\Lambda:=\varprojlim_n \Z_p[\Gamma/\Gamma^{p^n}].\] Setting $T$ to denote $(\gamma-1)$, we identify $\Lambda$ with the formal power series ring $\Z_p\llbracket T\rrbracket$. Let $G$ be the Galois group $\op{Gal}(L_\infty/F)$ and note that $G/\X$ is isomorphic to $\Gamma$. There is a natural action of $\Gamma$ on $\X$; for $g\in \Gamma$, pick a lift $\tilde{g}$ of $g$ to $G$, and for $x\in \X$, set $g\cdot x$ to denote $\tilde{g} x \tilde{g}^{-1}$. Note that $g\cdot x$ is contained in $\X$ since $\X$ is a normal subgroup of $G$, and also is independent of the choice of lift $\tilde{g}$ since $\X$ is abelian. Via this action, $\X$ is a module over the Iwasawa algebra $\Lambda$. 

\section{Iwasawa theory of constant $\Z_p$-extensions}\label{s 3}

\par In this section, we study the algebraic structure of the $\Lambda$-module $\X$ which was introduced in the previous section. Such results are analogous to the number field case, and in the function field case, some of the results may be known to experts. However, we find it of convenience to document them here. Let $D\subset G$ be the decomposition group of $\infty_F$. Since $\infty_F$ is totally inert in $F_\infty$ and $\infty_F$ is totally split in the extension $L_\infty/F_\infty$, the natural map $D\rightarrow \Gamma=G/X$ is an isomorphism. Hence, we find that $G=DX$.

\par Let $\sigma$ be the topological generator of $D$ which maps to $\gamma\in \Gamma$ w.r.t. the isomorphism $D\simeq \Gamma$.

\begin{lemma}
Let $G'$ be the closure of the commutator subgroup of $G$. Then, $G'=T\X$.
\end{lemma}
\begin{proof}
The proof is similar to that of \cite[Lemma 13.14]{washington1997introduction}; we provide details here. Note that an identification of $D$ with $\Gamma$ has been made. Let $a=\alpha x$ and $b=\beta y$ be elements of $G$, where $\alpha, \beta \in \Gamma$ and $x,y\in \X$. Given $\eta\in \Gamma$ and $v\in \X$, set $v^\eta$ to denote $\eta v \eta^{-1}$. A straightforward calculation shows that the commutator of $a$ and $b$ is given by
\[aba^{-1}b^{-1}=(x^\alpha)^{1-\beta} (y^\beta)^{\alpha-1}.\] Setting $\beta=1$ and $\alpha=\gamma$, we find that $T\cdot y=y^{\gamma-1}$ is contained in $G'$. Hence, we find that $T \X$ is contained in $G'$.

\par For the other inclusion, it suffices to observe that $(x^\alpha)^{1-\beta}$ and $(y^\beta)^{\alpha-1}$ are contained in $T \X$. This clear, since both $1-\beta$ and $\alpha-1$ are in the augmentation ideal of $\Lambda$, which is generated by $T$. 
\end{proof}

\begin{proposition}\label{control theorem}
With respect to notation above, we have the following isomorphism 
\[X_n\simeq \frac{\X}{\left((1+T)^{p^n}-1\right)\X}.\]
\end{proposition}

\begin{proof}
The proof is similar to that of \cite[Lemma 13.15]{washington1997introduction}, and the role of the inertia group in the number field context is interchanged with the group $D$.
\end{proof}

\begin{corollary}
With respect to notation above, $\X$ is a finitely generated $\Lambda$-module. 
\end{corollary}
\begin{proof}
By Nakayama's lemma, it suffices to show that $\X/\mathfrak{M}\X$ is finite, where $\mathfrak{M}=(p, T)$ is the maximal ideal of $\Lambda$. According to Proposition \ref{control theorem}, $\X/T\X$ is isomorphic to $X_0$, hence is finite. The result follows. 
\end{proof}

A map of $\Lambda$-modules is a \emph{pseudo-isomorphism} if its kernel and cokernel are finite. Given a finitely generated $\Lambda$-module $M$, there is a pseudo-isomorphism 
\[M\longrightarrow \Lambda^r \oplus \left(\bigoplus_{i=1}^s \frac{\Z_p\llbracket T\rrbracket}{p^{\mu_i}}\right)\oplus \left(\bigoplus_{i=1}^t \frac{\Z_p\llbracket T\rrbracket}{f_i(T)^{\lambda_i}}\right).\]
Here $f_i(T)$ are irreducible and distinguished polynomials in $\Z_p\llbracket T\rrbracket$. The $\mu$ and $\lambda$-invariants are defined by 
\[\text{setting }\mu(M):=\begin{cases} \sum_i \mu_i\text{ if }s>0,\\
0\text{ if }s=0,
\end{cases}
\text{ and, }
\lambda(M):=\begin{cases} \sum_i \lambda_i\op{deg}(f_i)\text{ if }t>0,\\
0\text{ if }t=0.
\end{cases}\]
We denote by $\mu_p(F)$ and $\lambda_p(F)$ the $\mu$ and $\lambda$-invariants of $\X$.

\begin{proposition}
Let $\mu=\mu_p(F)$ and $\lambda=\lambda_p(F)$, then there is an integer $\nu$ such that 
\[e_n=p^n \mu +n \lambda+\nu\] for all values $n\gg 0$. 
\end{proposition}

\begin{proof}
The result follows directly from \cite[Proposition 13.19, Lemma 13.21]{washington1997introduction} and Proposition \ref{control theorem}.
\end{proof}

The following result of Leitzel and Rosen shows that the $\mu$-invariant is always equal to $0$ for constant $\Z_p$-extensions.

\begin{theorem}[Leitzel, Rosen]
With respect to notation above, the module $\X$ is torsion over $\Lambda$, with $\mu_p(F)=0$.
\end{theorem}

\begin{lemma}\label{injectivity lemma}
Let $\op{Cl}_p(F_n)\rightarrow \op{Cl}_p(F_m)$ be the natural map for $m\geq n$. This map is injective.
\end{lemma}
\begin{proof}
Recall that $F_n$ was the composite $F\cdot \kappa_n^{(p)}=F\cdot \kappa_n^{(p)}(T)$. Since $\infty$ is unramified in $\bar{\kappa}(T)$ and totally ramified in $F$, it follows that $F$ and $\kappa_m^{(p)}(T)$ are linearly disjoint over $\kappa_n^{(p)}(T)$. Therefore, there are isomorphisms
\[\op{Gal}(F_m/F_n)\simeq \op{Gal}\left(\kappa_m^{(p)}(T)/\kappa_n^{(p)}(T)\right)\simeq \op{Gal}(\kappa_m^{(p)}/\kappa_n^{(p)}).\]Therefore (by Hilbert's theorem 90), we have that $H^1\left(F_m/F_n, (\kappa_m^{(p)})^\times\right)=0$. It suffices to show that the kernel of $\op{Cl}_p(F_n)\rightarrow \op{Cl}_p(F_m)$ is contained in $H^1\left(F_m/F_n, (\kappa_m^{(p)})^\times\right)$.

\par Let us define a map 
\[\Phi: \op{ker}\left(\op{Cl}_p(F_n)\rightarrow \op{Cl}_p(F_m)\right)\longrightarrow H^1\left(F_m/F_n, (\kappa_m^{(p)})^\times\right).\]

Given a divisor class $[D]$ in $\op{ker}\left(\op{Cl}_p(F_n)\rightarrow \op{Cl}_p(F_m)\right)$ represented by a divisor $D$, and $\sigma\in \op{Gal}(F_m/F_n)$, we find that $\sigma [D]-[D]=0$. We have that $D=\op{div}(\alpha)$ for a non-zero element $\alpha\in F_m$. Therefore, $\alpha^\sigma/\alpha$ has trivial divisor, and hence lies in $\cO_{F_m}^\times=(\kappa_m^{(p)})^{\times}$. We set $\Phi(D)(\sigma):=\alpha^\sigma/\alpha$ and show that $\Phi$ is injective. Suppose that $\Phi(D)=0$. Then, there is $\epsilon\in (\kappa_m^{(p)})^{\times}$ such that for all $\sigma\in \op{Gal}(F_m/F_n)$, 
\[\alpha^\sigma/\alpha=\epsilon^\sigma/\epsilon.\] Replacing $\alpha$ by $\alpha/\epsilon$, we find that $\alpha^\sigma=\alpha$ for all $\sigma$. Thus, $\alpha$ must belong to $F_n$, and the ideal class of $D$ in $\op{Cl}_p(F_n)$ is trivial.
\end{proof}

\begin{theorem}\label{finite submodules}
The module $\X$ does not contain any non-trivial and finite $\Lambda$-modules. 
\end{theorem}

\begin{proof}
The result follows from Lemma \ref{injectivity lemma} by an argument identical to that of \cite[Proposition 13.28]{washington1997introduction}.
\end{proof}

\section{Proof of the main results}\label{s 4}
\par Let $p$ and $\ell$ be prime numbers and assume that $p$ is odd. Let $q$ be a power of $\ell$ and $\F_q$ be the finite field with $q$ elements. Denote by $\F_q(T)$ the field of rational functions in one variable over $\F_q$. In this section, we recall the main result of \cite{ellenberg2016homological} and derive consequences for the distribution of the $\lambda$-invariant $\lambda_p(F)$ for the family of imaginary quadratic field extensions $F$ of $\F_q(T)$.

\par Let $\cF$ be a family of imaginary quadratic function field extensions of $\F_q(T)$ ordered by disciminant. For $x\in \mathbb{R}_{>0}$, set $\cF_q(x)$ to be the subset of $\cF$ consisting of imaginary quadratic fields with discriminant $\leq x$. Let $A$ be a finite abelian $p$-group, and denote by $\cF_q(x;A)$ the subset of $\cF(x)$ consisting of imaginary quadratic fields $F$ with $\op{Cl}_p(F)$ isomorphic to $A$. Set $\delta^{\pm}_A(q)$ to denote the upper and lower densities defined as follows
\[\delta^+_A(q):=\limsup_{x\rightarrow \infty} \frac{\# \cF_q(x;A)}{\# \cF_q(x)}\text{ and }\delta^-_A(q):=\liminf_{x\rightarrow \infty} \frac{\# \cF_q(x;A)}{\# \cF_q(x)}.\]

\begin{theorem}[Ellenberg-Venkatesh-Westerland]\label{EVW thm}
Let $p$ be an odd prime and $A$ be a finite abelian $p$-group. Then, as $q\not \equiv 1\mod{p}$ goes to $\infty$, the densities $\delta^+(q)$ and $\delta^-(q)$ both converge to $\frac{\prod_{i\geq 1}\left(1-q^{-i}\right)}{\# \op{Aut}(A)}$.
\end{theorem}

Given any finite abelian group $A$, the $p$-rank is the dimension of $A/pA$ over $\Z/p\Z$.
Let $\delta_r^{-}(q)$ be the lower density of imaginary quadratic extensions of $\F_q(T)$ for which the $p$-rank of the class group is $\geq r$. We have the following corollary.

\begin{corollary}\label{cor 4.2} Let $p$ be an odd prime number and $r>0$ be an integer. Then, there is a constant $c(p, r)>0$ depending only on $p$ and $r$ such that for all values of $q>c(p, r)$ satisfying $q\not\equiv 1\mod{p}$, we have that
$\delta_r^{-}(q)>0$.
\end{corollary}

\begin{proof}
Let $A$ be any choice of abelian $p$-group with $p$-rank $\geq r$. Then, by Theorem \ref{EVW thm}, we find that $\delta_A^{-}(q)>0$ for sufficiently large values of $q\not \equiv 1\mod{p}$. Clearly, $\delta_r^{-}(q)\geq \delta_A^{-}(q)$ since $A$ has $p$-rank $\geq r$, and the result follows.
\end{proof}
Let $r$ be a positive integer and $\delta^{-}(q, \lambda\geq r)$ be the lower density of imaginary quadratic fields with $\lambda_p(F)\geq r$. More precisely, let $\cF_q(x;\lambda\geq r)$ be the subset of $\cF_q(x)$ consisting of all $F$ such that $\lambda_p(F)\geq r$. Then, $\delta^{-}(q, \lambda\geq r)$ is given by the following lower limit
\[\delta^{-}(q, \lambda\geq r):=\liminf_{x\rightarrow \infty} \frac{\# \cF_q(x;\lambda\geq r)}{\# \cF_q(x)}.\]
Denote by $r_p(F)$ the $p$-rank of $\op{Cl}(F)$.
 Before proving the main Theorem, we establish a relationship between $\lambda_p(F)$ and $r_p(F)$.

\begin{lemma}\label{p rank and lambda} Let $p$ be an odd prime and $F$ be an imaginary quadratic field. Then, we have that
$\lambda_p(F)\geq r_p(F)$.
\end{lemma}

\begin{proof}
Let $\mathfrak{M}=(p, T)$ be the maximal ideal of $\Lambda$ and set $g_p(F)$ to denote the dimension of $\X/\mathfrak{M}$ over $\Lambda/\mathfrak{M}=\Z/p\Z$. Since by Theorem \ref{finite submodules}, $\X$ does not contain any non-trivial finite $\Lambda$-submodules, we have therefore by \cite[Lemma 2.2]{matsuno2007construction} that $\mu_p(F)+\lambda_p(F)\geq g_p(F)$. Since the $\mu$-invariant vanishes, we find that $\lambda_p(F)\geq g_p(F)$. Since $\X$ surjects onto $X_0$, it follows that $\X/\mathfrak{M}$ surjects onto $X_0/p$; the action of $T$ on $X_0$ being trivial. Therefore, $g_p(F)\geq r_p(F)$, and thus the result follows.
\end{proof}

\begin{theorem}\label{main theorem}
Let $p$ be an odd prime number and $r>0$ be an integer. Then, there is a constant $c(p, r)>0$ depending only on $p$ and $r$ such that for all values of $q>c(p, r)$ satisfying $q\not\equiv 1\mod{p}$, we have that $\delta^{-}(q, \lambda\geq r)>0$.
\end{theorem}

\begin{proof}
It follows from Lemma \ref{p rank and lambda} that $\delta^{-}(q, \lambda\geq r)\geq \delta_r^-(q)$, and the result thus follows from Corollary \ref{cor 4.2}.
\end{proof}

\bibliographystyle{alpha}
\bibliography{references}
\end{document}